\documentclass[10pt]{article}

\font\smallit=cmti10

\usepackage{latexsym,amsmath,epsfig,amsthm}
\usepackage{MnSymbol}

\usepackage{mathtools}
\usepackage{blkarray}
\usepackage[table]{xcolor}

\newtheorem{thm}{\bfseries Theorem}
\newtheorem{remark}[thm]{\bfseries Remark}    %%   are numbered consecutively
\newtheorem{cor}[thm]{\bfseries Corollary}     

\newtheorem{cl}[thm]{\bfseries Claim}
\newtheorem{obs}[thm]{\bfseries Observation}

\newtheorem{conj}[thm]{\bfseries Conjecture}

\newtheorem{const}[thm]{\bfseries Construction}

\makeatletter

\renewcommand\section{\@startsection {section}{1}{\z@}
{-30pt \@plus -1ex \@minus -.2ex}
{2.3ex \@plus.2ex}
{\normalfont\normalsize\bfseries}}

\renewcommand\subsection{\@startsection{subsection}{2}{\z@}
{-3.25ex\@plus -1ex \@minus -.2ex}
{1.5ex \@plus .2ex}
{\normalfont\normalsize\bfseries}}

\renewcommand{\@seccntformat}[1]{\csname the#1\endcsname. }

\newcommand\bovermat[2]{%
  \makebox[0pt][l]{$\smash{\overbrace{\phantom{%
    \begin{matrix}#2\end{matrix}}}^{\text{#1}}}$}#2}

\newcommand\bundermat[2]{%
  \makebox[0pt][l]{$\smash{\underbrace{\phantom{%
    \begin{matrix}#2\end{matrix}}}_{\text{#1}}}$}#2}

\newcommand\lep{\cellcolor{black!10}}
\newcommand\lepc{\cellcolor{black!20}}
\newcommand\lepcs{\cellcolor{black!30}}
\newcommand\lepcso{\cellcolor{black!40}}

\def\zo {$0\text{-}1$\ }
\def\turn{\text{turn}}
\def\set#1{\{#1\}}

\def\jl{$\righthalfcap$}
\def\lj{$\lefthalfcup$}
\def\szumma{st_{\Sigma}}

\begin{document}

%Next comes your title and author list information in the following construct:

\begin{center}
\uppercase{\bf On the staircases of Gy\'arf\'as}
\vskip 20pt
{\bf J\'anos Cs\'anyi}%\footnote{any footnote here}}
     \\
{\smallit University of Szeged, Bolyai Institute
%         Aradi V\'ertan\'uk tere 1., Szeged, Hungary 6720
          }\\
{\tt csanyi1990@gmail.com}\\
\vskip 10pt
{\bf P\'eter Hajnal}%\footnote{any footnote here}}
     \\
{\smallit University of Szeged, Bolyai Institute,
%         Aradi V\'ertan\'uk tere 1., Szeged, Hungary 6720\\
          and\\
          Alfr\'ed R\'enyi Institute of Mathematics,
          Hungarian Academy of Sciences}\\ 
{\tt hajnal@math.u-szeged.hu}\\
\vskip 10pt
{\bf G\'abor V.~Nagy}%\footnote{any footnote here}}
     \\
{\smallit University of Szeged, Bolyai Institute
%          Aradi V\'ertan\'uk tere 1., Szeged, Hungary 6720
          }\\
{\tt ngaba@math.u-szeged.hu}\\
\end{center}

\vskip 30pt

%\centerline{\smallit Received: , Revised: , Accepted: , Published: } 
% We will fill in the dates
%\vskip 30pt

%Followed by your abstract in the following construct:

\centerline{\bf Abstract}
Gy\'arf\'as in \cite{Gy} investigated a geometric Ramsey problem
on convex, separated, balanced, geometric $K_{n,n}$.
This led to appealing
extremal problem on square \zo matrices.
Gy\'arf\'as conjectured that any \zo matrix of size $n\times n$ 
has a staircase of size $n-1$.

We introduce the non-symmetric version of Gy\'arf\'as' problem.
We give upper bounds and in certain range matching lower bound on
the corresponding extremal function.
In the square/balanced case we improve the
$(4/5+\epsilon)n$ lower bound of Cai, Gy\'arf\'as et al.~\cite{Gy2}
to $5n/6-7/12$.
We settle the problem when instead of considering
maximum staircases we deal with the sum of the
size of the longest $0$- and $1$-staircases.

\noindent

\section{
                Introduction
}

It is well-known
(an early remark of Erd\H{o}s and Rado, nowadays a standard
introductory exercise in graph theory courses)
that in any red/blue edge coloring of the
complete graph a monochromatic spanning tree is guaranteed.
It was in 1998 when K\'arolyi, Pach and T\'oth \cite{KPT}
proved the geometric version
of this fact: We take $n$ independent
points on the plane, connect all the pairs with a line segment,
and color them arbitrarily with red and blue colors.
Then a non-crossing spanning tree is guaranteed.

The bipartite case of the original
graph theoretical question is easy:
If we red/blue color the edges of $K_{n,n}$
(the balanced complete bipartite graph) then
the largest monochromatic tree has at least $n$ vertices if $n$ is even
and $n+1$ if $n$ is odd.
Furthermore this bound is optimal, certain coloring
doesn't contain a larger
monochromatic tree.

The beautiful theorem
from \cite{KPT} led Gy\'arf\'as to consider the
geometric problem when the underlying graph is
a complete bipartite graph:
Take any $2n$ points in convex position on the plane.
Assume that a line separates them into two groups
of $n$ points.
Connect all crossing pairs with a line segment
and color them arbitrarily with red/blue colors.
What is the size of the largest monochromatic tree
that you can guarantee?

The above/graph version of the
basic question can be easily transformed to 
a matrix version as it was done in \cite{Gy2}.

Let $M$ be a \zo matrix. A $0$-staircase is a sequence $\set{s_i}_{i=1}^\ell$ of
zeroes in $M$ such that $s_{i+1}$ is either to the right of $s_i$ in the same row,
or it is below $s_i$ in the same column ($i=1,2,\ldots,\ell-1$).
(We emphasize that $s_i$ and $s_{i+1}$ do not have to be neighbouring elements in $M$.)
A $1$-staircase is defined similarly on ones of $M$. $M$ is a homogeneous
staircase in $M$ iff it is either a $0$- or a $1$-staircase. The length
of a homogeneous staircase $S$ is the number of elements in
it, we denote it as $|S|$. $S$ can be viewed as $|S|-1$ steps, where the steps can be right steps or down steps,
formed by the consecutive elements of $S$.
An element of $S$ is a turning point, if it is involved in two steps with
different directions.
$\turn(S)$ denotes the number of turns in $S$.
For example $\turn (S)=0$ iff $S$ includes identical elements from a row, 
or from a column.
$\turn(S)=1$ iff the steps of $S$ are some horizontal (staying in the same row)
steps followed by some vertical steps or vice versa.
In the first case we say that $S$ is a \jl-staircase.
The first elements forming the horizontal steps will be referred as the
horizontal hand of our \jl-staircase.
The last block of elements that form the vertical steps 
will be referred as the
vertical hand of our \jl-staircase.
In the second case
we say that $S$ is a \lj-staircase.
We can use the notation of horizontal/vertical 
hands in this case too.

\begin{conj}[Gy\'arf\'as' Conjecture]
Any \zo matrix of size $n\times n$
contains a staircase of size $n-1$.
\end{conj}

Let $st(M)$ be the maximum among the lengths of the
homogeneous staircases of $M$.
Let $st_0(M)$ (resp. $st_1(M)$) be the maximum among the lengths of the
homogeneous $0$-staircases (resp.~$1$-staircases) of $M$.
Hence $st(M)=\max\set{st_0(M),st_1(M)}$.

We define a symmetric and an asymmetric version
of the original Ramsey-type question.
\[
st(n)=\min\set{st(M): M\in\set{0,1}^{n\times n}},
\]
\[
st(n,N)=\min\set{st(M): M\in\set{0,1}^{n\times N}}.
\]
In the asymmetric case it is obvious that $st(n,N)=st(N,n)$.
We always assume that $N\geq n$.

The Gy\'arf\'as' conjecture states that $st(n)\geq n-1$.
Since \cite{Gy2} contains an easy example that shows 
$st(n)\leq n-1$ (assuming $n>1$) we can state
Gy\'arf\'as' conjecture as $st(n)=n-1$ if $n>1$.

Instead of $st(M)=\max\set{st_0(M),st_1(M)}$ one can
investigate
\[
\szumma(M)=st_0(M)+st_1(M).
\]

It is natural to introduce
\[
\szumma(n)=\min\set{\szumma(M): M\in\set{0,1}^{n\times n}},
\]
\[
\szumma(n,N)=\min\set{\szumma(M): M\in\set{0,1}^{n\times N}}.
\]

Cai, Grindstaff, Gy\'arf\'as and Shull stated a conjecture
concerning the function
$\szumma(n)$ (see \cite{Gy0}), that turned out to be false (see
the final remark in \cite{Gy2}, the journal version of \cite{Gy0}).
In section 2 we determine the exact value
of 
$\szumma(n, N)$.
In section 3 we deal with $st(n,N)$.
We present two constructions (hence we give upper bounds
on $st(n,N)$). We conjecture that our matrices gives the right value of $st(n,N)$.
Our conjecture is consistent with Gy\'arf\'as' conjecture.
We will be able to prove our conjecture in certain range.
Unfortunately the case of square and ``near-square''
matrices is still unsolved.

In the final section we give
a significant improvement of the 
$\frac{4}{5}n$ bound in \cite{Gy} (as opposed to the one
in \cite{Gy2}). 
We are going to prove the following theorem.

\begin{thm}
For any $M\in\set{0,1}^{n\times n}$ 
\[
st(M)\geq\frac{5}{6}n-\frac{7}{12}.
\]
\end{thm}

\section{
                The exact value of $\szumma(n,N)$
}

We start with an easy construction.

\begin{const}
Let $P^{(n,N)}\in\set{0,1}^{n\times N}$
be the matrix, where we have $P^{(n,N)}_{i,j}=0$ 
iff $i+j\leq\lfloor\frac{n}{2}\rfloor
+1$ or $i+j\geq\lfloor\frac{n}{2}\rfloor+N+1$.
\end{const}

The following two examples help to understand the formalism:

\[
P^{(6,8)}=\begin{pmatrix}
 0 & 0 & 0 & 1 & 1 & 1 & 1 & 1 \\
 0 & 0 & 1 & 1 & 1 & 1 & 1 & 1 \\
 0 & 1 & 1 & 1 & 1 & 1 & 1 & 1 \\
 1 & 1 & 1 & 1 & 1 & 1 & 1 & 0 \\
 1 & 1 & 1 & 1 & 1 & 1 & 0 & 0 \\
 1 & 1 & 1 & 1 & 1 & 0 & 0 & 0 \\
\end{pmatrix},
\qquad
P^{(7,8)}=\begin{pmatrix}
 0 & 0 & 0 & 1 & 1 & 1 & 1 & 1 \\
 0 & 0 & 1 & 1 & 1 & 1 & 1 & 1 \\
 0 & 1 & 1 & 1 & 1 & 1 & 1 & 1 \\
 1 & 1 & 1 & 1 & 1 & 1 & 1 & 0 \\
 1 & 1 & 1 & 1 & 1 & 1 & 0 & 0 \\
 1 & 1 & 1 & 1 & 1 & 0 & 0 & 0 \\
 1 & 1 & 1 & 1 & 0 & 0 & 0 & 0 \\
\end{pmatrix}  
\]
Note that the $0$'s form two triangular regions in our matrix
(in the upper left and the lower right corners).
Because of our convention ($n\leq N$)
there is no $0$-staircase that contains $0$'s from both triangular regions.
Hence the following fact is easy:

\begin{obs} For any $n\leq N$,
\[
st_0(P^{(n,N)})=\left\lceil\frac{n}{2}\right\rceil,
\qquad 
st_1(P^{(n,N)})=N-1,
\] and so
%st_0(P^{(n,N)})+st_1(P^{(n,N)})=\left\lceil\frac{n}{2}\right\rceil+N-1.
\[\szumma(P^{(n,N)})=\left\lceil\frac{n}{2}\right\rceil+N-1.
\]
\end{obs}

We are going to prove a matching lower bound on $\szumma(M)$
for any $M\in\set{0,1}^{n\times N}$.

\begin{thm} For any $n\leq N$,
\[
\szumma(n,N)=\left\lceil\frac{n}{2}\right\rceil+N-1.
\]
\label{tetel_osszeg}\end{thm}

\begin{proof}
Let $M\in\set{0,1}^{n\times N}$ arbitrary.
Take the elements of the first column, that give the majority of this
column. We can assume that they have the common value $1$.
Let $a$ be the position of the lowest $1$ in the first column.
The number of $1$'s at and above $a$ is at least $\lceil\frac{n}{2}\rceil$.
Let $S_1$ the \lj-staircase centered at $a$. Let $L_2$ 
the staircase formed by the $0$'s in the row of $a$.
It is easy to see that
\[
|L_1|+|L_2|\geq \left\lceil\frac{n}{2}\right\rceil+N-1.
\]
This observation proves the theorem.
\end{proof}

\section{
                An upper bound for asymmetric matrices
}

In this section we give an upper bound on $st(n,N)$. We will easily see that 
our bound is the truth for ``wide'' enough matrices. (The sharpness of the bound
has been validated by computer for a few additional dimensions, too,
but only for small $n$ and $N$ values.) We exhibit two constructions. They correspond
to two different ranges of shapes.

\begin{const} For $N\geq\lfloor\frac{5}{2}n\rfloor-1$, we define the matrix $Q^{(n,N)}$ as follows:
Let $Q^{(n,N)}_{i,j}=1$ iff one of the following three possibilities holds:
\begin{itemize}
\item[(1)]
$i\leq \left\lceil \frac n2\right\rceil$ and $i+j\leq \left\lceil \frac n2\right\rceil+1$,
\item[(2)]
$i\leq \left\lceil \frac n2\right\rceil$ and $\left\lceil \frac n2\right\rceil +\left\lfloor\frac{\lceil n/2\rceil+N-1}2\right\rfloor +1< i+j$,
\item[(3)]
$i>\left\lceil \frac n2\right\rceil$ and $\left\lceil\frac n2\right\rceil+N-\left\lfloor\frac{\lceil n/2\rceil+N-1}2\right\rfloor < i+j\leq\left\lceil\frac n2\right\rceil+N$.
\end{itemize}
\end{const}
The reader is advised to check the example below to understand the structure of matrix $Q^{(n,N)}$:
The first $\lceil n/2\rceil$ rows of $Q^{(n,N)}$ forms an ``isosceles right triangle'' of 1's
in the upper left corner (the length of the legs is $\lceil n/2\rceil$), followed by a ``parallelogram'' of 0's
with horizontal side of length $\left\lfloor\frac{\lceil n/2\rceil+N-1}2\right\rfloor$, and the rest of this upper region is filled with 1's.
The last $\lfloor n/2\rfloor$ rows are defined analogously: the triangle in the lower right corner has leg length $\lfloor n/2\rfloor$,
and the parallelogram of 1's has the same width as the parallelogram of 0's above.
\setcounter{MaxMatrixCols}{30}
\[
Q^{(6,19)}=
\begin{matrix}
\phantom{0}\\
\phantom{0}\\
\begin{pmatrix}
\bovermat{$\left\lceil \frac n2\right\rceil$}{1&1&1}&\bovermat {$\left\lfloor\frac{\lceil n/2\rceil+N-1}2\right\rfloor$}{0&0&0&0&0&0&0&0&0&0}&1&1&1&1&1&1\\
1&1&0&0&0&0&0&0&0&0&0&0&1&1&1&1&1&1&1\\
1&0&0&0&0&0&0&0&0&0&0&1&1&1&1&1&1&1&1\\
0&0&0&0&0&0&0&0&1&1&1&1&1&1&1&1&1&1&0\\
0&0&0&0&0&0&0&1&1&1&1&1&1&1&1&1&1&0&0\\
0&0&0&0&0&0&\bundermat {$\left\lfloor\frac{\lceil n/2\rceil+N-1}2\right\rfloor$}{1&1&1&1&1&1&1&1&1&1}&\bundermat{$\left\lfloor \frac n2\right\rfloor$}{0&0&0}\\
\end{pmatrix}\\
\phantom{0}\\
\phantom{0}\\
\end{matrix}
\begin{matrix}
\left.\begin{matrix}
{}\\{}\\{}
\end{matrix}\right\}\left\lceil \frac n2\right\rceil\\
\left.\begin{matrix}
{}\\{}\\{}
\end{matrix}\right\}\left\lfloor \frac n2\right\rfloor
\end{matrix}
\]
\begin{cl} For all $N\geq\lfloor\frac{5}{2}n\rfloor-1$,
\[
st(Q^{(n,N)})=\left\lceil\frac{\lceil n/2\rceil+N-1}2\right\rceil.
\]
%and so
%\[
%st(n,N)\leq\left\lceil\frac{\lceil n/2\rceil+N-1}2\right\rceil.
%\]
\label{konstrukcio1}\end{cl}
\begin{proof} The 1's in the first row and last column form a 1-staircase of length $\left\lceil\frac{\lceil n/2\rceil+N-1}2\right\rceil$.
We show that there are no longer homogeneous staircases in $Q^{(n,N)}$. We prove this for 1-staircases only, the case of 0-staircases is analogous.\par
The $\lceil n/2\rceil$-th column, i.e.\ the column of the rightmost 1 in the upper left corner, precedes the column of the first 1 of the last row.
(This can be seen by verifying that
\[
\left\lceil \frac n2\right\rceil+\left\lfloor\frac{\lceil n/2\rceil+N-1}2\right\rfloor+\left\lfloor \frac n2\right\rfloor \leq N
\]
holds for $N\geq\lfloor\frac{5}{2}n\rfloor-1$.) This means that every 1-staircase that starts from the upper left triangle,
can only leave this triangle with a right step; which implies that if we translate this triangle of 1's to the right without overlapping any other 1's,
the maximal length of 1-staircases does not decrease. We can translate this triangle to the right by $\left\lfloor\frac{\lceil n/2\rceil+N-1}2\right\rfloor$ positions
(the width of the upper parallelogram of 0's) without overlapping. It is easy to check that in the obtained matrix $\widetilde Q$, all the 1's lies in the region
\[
\left\lceil\frac n2\right\rceil+N-\left\lceil\frac{\lceil n/2\rceil+N-1}2\right\rceil+1\leq i+j\leq \left\lceil\frac n2\right\rceil+N.
\]
Since in a staircase each step increases $i+j$, the sum of the ``coordinates'' of the actual position, it is obvious that there is no
1-staircase in $\widetilde Q$ with length greater than $\left\lceil\frac{\lceil n/2\rceil+N-1}2\right\rceil$, as required.
\end{proof}

Theorem~\ref{tetel_osszeg} says that $st_0(M)+st_1(M)\geq\lceil\frac n2\rceil+N-1$, and so $st(M)\geq\left\lceil\frac{\lceil n/2\rceil+N-1}2\right\rceil$ holds
for every $M\in\set{0,1}^{n\times N}$. Hence we always have $st(n,N)\geq\left\lceil\frac{\lceil n/2\rceil+N-1}2\right\rceil$.
Claim~\ref{konstrukcio1} says that $st(n,N)\leq\left\lceil\frac{\lceil n/2\rceil+N-1}2\right\rceil$, for all $N\geq\lfloor\frac{5}{2}n\rfloor-1$, so we obtained the following:
\begin{thm} For all $N\geq\lfloor\frac{5}{2}n\rfloor-1$,
\[
st(n,N)=\left\lceil\frac{\lceil n/2\rceil+N-1}2\right\rceil.
\]
\end{thm}
\def\rondasag{\left\lceil\frac{2n+N-2}3\right\rceil}%
\begin{const}
For $n<N<\lfloor\frac{5}{2}n\rfloor-1$, we define the matrix $R^{(n,N)}$ as follows:
Let $R^{(n,N)}_{i,j}=1$ iff one of the following three possibilities holds:
\begin{itemize}
\item[(1)]
$i\leq \left\lceil \frac n2\right\rceil$ and $i+j\leq\left\lfloor\frac{N-n+2}3\right\rfloor+1$,
\item[(2)]
$i\leq \left\lceil \frac n2\right\rceil$ and $\left\lfloor\frac{N-n+2}3\right\rfloor+\rondasag+1 < i+j$,
\item[(3)]
$i>\left\lceil \frac n2\right\rceil$ and $n+\left\lfloor\frac{N-n+2}3\right\rfloor < i+j\leq n + N - \left\lceil\frac{N-n-1}3\right\rceil$.
\end{itemize}
\end{const}
\[
R^{(10,18)}=
\begin{matrix}
\phantom{0}\\
\phantom{0}\\
\begin{pmatrix}
\bovermat {$\left\lfloor\frac{N-n+2}3\right\rfloor$}{1&1&1}&\bovermat {$\rondasag$}{0&0&0&0&0&0&0&0&0&0&0&0}&\bovermat {$\left\lceil\frac{N-n-1}3\right\rceil$}{1&1&1}\\
1&1&0&0&0&0&0&0&0&0&0&0&0&0&1&1&1&1\\
1&0&0&0&0&0&0&0&0&0&0&0&0&1&1&1&1&1\\
0&0&0&0&0&0&0&0&0&0&0&0&1&1&1&1&1&1\\
0&0&0&0&0&0&0&0&0&0&0&1&1&1&1&1&1&1\\
0&0&0&0&0&0&0&1&1&1&1&1&1&1&1&1&1&1\\
0&0&0&0&0&0&1&1&1&1&1&1&1&1&1&1&1&1\\
0&0&0&0&0&1&1&1&1&1&1&1&1&1&1&1&1&0\\
0&0&0&0&1&1&1&1&1&1&1&1&1&1&1&1&0&0\\
\bundermat {$\left\lfloor\frac{N-n+2}3\right\rfloor$}{0&0&0}&\bundermat {$\rondasag$}{1&1&1&1&1&1&1&1&1&1&1&1}&\bundermat {$\left\lceil\frac{N-n-1}3\right\rceil$}{0&0&0}
\end{pmatrix}\\
\phantom{0}\\
\phantom{0}
\end{matrix}
\begin{matrix}
\left.\begin{matrix}
{}\\{}\\{}\\{}\\{}
\end{matrix}\right\}\left\lceil \frac n2\right\rceil\\
\left.\begin{matrix}
{}\\{}\\{}\\{}\\{}
\end{matrix}\right\}\left\lfloor \frac n2\right\rfloor
\end{matrix}
\]
The reader should check that $\left\lfloor\frac{N-n+2}3\right\rfloor + \rondasag + \left\lceil\frac{N-n-1}3\right\rceil$ is always equal to~$N$,
and the conditions imply that $\left\lfloor\frac{N-n+2}3\right\rfloor\leq \left\lceil\frac n2\right\rceil$ and $\left\lceil\frac{N-n-1}3\right\rceil\leq\left\lfloor \frac n2\right\rfloor$,
as the above example suggests. After doing this and reading the proof of Claim~\ref{konstrukcio1}, it should be straightforward to verify the following:
\begin{cl}
For all $n<N<\lfloor\frac{5}{2}n\rfloor-1$,
\[
st(R^{(n,N)})=\rondasag.
\]
\end{cl}
\begin{cor}
For all $n<N<\lfloor\frac{5}{2}n\rfloor-1$,
\[
st(n,N)\leq\rondasag.
\]
\end{cor}
\section{
                Improved lower bound, the proof of Theorem 2
}

Let $M$ be an arbitrary \zo matrix of size $n\times n$.
Let $a_1$ be the last element of the first row.
Without loss of generality we can assume that $a_1=1$.
Let $a_2$ be the last $0$ in the first row.
Let $a_3$ be the top $0$ (the $0$ with the least first/row index)
in the last column.
Let $a_4$ denote the element/position in
the intersection of the column of $a_2$ and the row of $a_3$.

It is possible that $a_2$ is not well 
defined (the first row does not contain any $0$).
Since in this case Gy\'arf\'as' conjecture is obviously
true we assume that $a_2$ (and $a_3$ as well) is well defined.

\begin{equation*}
  \begin{blockarray}{ccccccc}
& & & & &  & \\
\begin{block}{(cccccc)c}
\bovermat{$x_1$ many $1$'s}{
        & \cdots}& 0_{a_2}      & 1 & \cdots & 1_{{a_1}}     & \\
        &        &             &   &        & 1           & \\
\vdots  &        & \vdots      &   &        & \vdots      & \\
        &        &             &   &        & 1           & \\
        &        &     ?_{a_4}  &   &        & 0_{a_3       } & \\
\vdots  &        & \vdots      &   &        & \vdots      & \}\text{ $w_1$ many $1$'s}\\
\end{block}  
\end{blockarray}
\end{equation*}
\begin{center}\begin{minipage}[c]{4in}
\footnotesize
The summary of our notations, with some references to future
parameters.
\end{minipage}\end{center}\smallskip

\begin{obs}
If $a_4=0$ then Gy\'arf\'as' conjecture is true for our
matrix.
\end{obs}

\begin{proof}
We consider two staircases: (1) The one that is
formed by the $1$'s in the first row and the $1$'s in the last column.
(2) The one that is formed by the $0$'s of the first row, $a_4$,
the $0$'s above and on its right side of $a_4$, and the $0$'s of the last column.
The staircase described in (1) is a $1$-staircase,
and (2) defines a $0$-staircase.

\begin{equation*}
\left(\begin{array}{cccccc}
\lepc  &\lepc \cdots &\lepc 0           &\lepc 1 &\lepc \cdots &\lepc 1\\
       &             &                  &        &             &\lepc 1 \\
\vdots &             & \vdots           &        &             &\lepc \vdots \\
       &             &                  &        &             &\lepc 1      \\
       &             &     0            &        &             &\lepc 0 \\
\vdots &             &\vdots            &        &             &\lepc \vdots \\
  \end{array}\right)
\qquad\qquad\qquad
\left(\begin{array}{cccccc}
\lep        &\lep \cdots &\lep 0      & 1 & \cdots & 1      \\
            &            &\lep        &   &        & 1      \\
\vdots      &            &\lep \vdots &   &        & \vdots \\
            &            &\lep        &   &        & 1      \\
            &            &\lep 0      &\lep &\lep  &\lep 0  \\
\vdots      &            &\vdots      &     &      &\lep \vdots\\
  \end{array}\right)
\end{equation*}
\begin{center}\begin{minipage}[c]{4in}
\footnotesize
The highlighted elements just show the shape of the staircase. For
the staircase we must consider only the elements with the same
value as the turning points have.
\end{minipage}\end{center}\smallskip

It is easy to check that the sum of the length of the
above two staircases is at least $2n$. Hence the 
longer one proves the conjecture.
\end{proof}

In the rest of the proof we assume that
$a_4=1$.
Let $S_i$ denote the \jl-staircase centered at $a_i$
($i=1,2,3,4$).
Let $s_i=|S_i|$.

\begin{equation*}
S_1=  \left(\begin{array}{cccccc}
\lepc  &\lepc \cdots &\lepc 0 &\lepc 1 &\lepc \cdots &\lepc 1 \\
       &             &                  &        &             &\lepc 1 \\
\vdots &             & \vdots           &        &             &\lepc \vdots \\
       &             &                  &        &             &\lepc 1      \\
       &             &     1            &        &        &\lepc 0 \\
\vdots &             &\vdots            &        &        &\lepc \vdots      \\
  \end{array}\right)
\qquad\qquad\qquad
S_2=  \left(\begin{array}{cccccc}
\lep        &\lep \cdots &\lep 0 & 1 & \cdots & 1           \\
            &            &\lep             &   &        & 1           \\
\vdots      &            &\lep \vdots      &   &        & \vdots      \\
            &            &\lep             &   &        & 1           \\
            &            &\lep     1       &   &        & 0 \\
\vdots      &            &\lep \vdots      &   &        & \vdots      \\
  \end{array}\right)
\end{equation*}

\begin{equation*}
S_3=  \left(\begin{array}{cccccc}
        & \cdots & 0      & 1      & \cdots & 1                  \\
        &        &                  &        &        & 1                  \\
\vdots  &        & \vdots           &        &        &\vdots              \\
        &        &                  &        &        &1                   \\
\lepcso &\lepcso &\lepcso   1       &\lepcso &\lepcso &\lepcso 0 \\
\vdots  &        &\vdots            &        &        &\lepcso \vdots      \\
  \end{array}\right)
\qquad\qquad\qquad
S_4=  \left(\begin{array}{cccccc}
            & \cdots     & 0       & 1 & \cdots & 1           \\
            &            &                   &   &        & 1           \\
\vdots      &            & \vdots            &   &        & \vdots      \\
            &            &                   &   &        & 1           \\
\lepcs      &\lepcs      &\lepcs     1       &   &        & 0 \\
\vdots      &            &\lepcs \vdots      &   &        & \vdots      \\
  \end{array}\right)
\end{equation*}

\begin{obs}
\[
2s_1+s_2+s_3+s_4=4n-3+x_1+y_0+z_0+w_1.
\]
where
$x_1$ denotes the number of $1$'s in the first row before $a_2$,
%$y_0$ denotes the number of $0$'s in the column of $a_2$ above $a_4$,
$y_0$ denotes the number of $0$'s in the column of $a_2$ between $a_2$ and $a_4$,
%$z_0$ denotes the number of $0$'s in the row of $a_3$ after $a_4$,
$z_0$ denotes the number of $0$'s in the row of $a_3$ between $a_4$ and $a_3$,
and
$w_1$ denotes the number of $1$'s in the last column under $a_3$.
\end{obs}

\begin{proof}
$s_1$, the number of $0$'s in the horizontal hands of $S_2$,
and the number of $0$'s in the vertical hands of $S_3$
add up to $2n-1$.

$s_4$, the number
the number of $0$'s before $a_4$,
and the number of $0$'s under $a_4$ (these $0$'s are counted
in $s_2$ and $s_3$)
add up to the number of the shaded positions in the figure of $S_4$.
The second $s_1$ term counts the $1$'s in last block of $1$'s of the first row 
and the $1$'s in the top block of $1$'s of the last column.
The considered contribution of the sum of length give us
$2n-2$.

It is easy to see that
the not counted contribution of the staircases
gives us the last four terms on the right hand side.
\end{proof}

We can repeat the same argument with exchanging the role
of rows and columns.
Then the role of $a_1$ will be played by the last element of the first column
(i.e.~the first element of the last row). Let $a$ be the value of this element.
$\overline a$ denotes $1-a$.

\begin{equation*}
  \begin{blockarray}{ccccccc}
\begin{block}{c(cccccc)}
\text{$w_a'$ many $a$'s }\{
 &\vdots &          &    &\vdots       &            &\vdots \\
 &\overline a_{{a_3}'}&\hdots&    & a_{{a_4}'}           & \hdots     &            \\
 & a  &             &    &             &            & \\
 &\vdots &          &    &\vdots       &            &\vdots \\
 & a  &             &    &             &            & \\
 & a_{{a_1}'}  & \hdots      & a  &\overline a_{{a_2}'}  \bundermat{\text{$x_a'$ many $a$'s}}{& \hdots} &\\
\end{block}  
 &    &             &    &             &            & \\
\end{blockarray}
\end{equation*}

All previous notations can be introduced in this setting
(we use the same letters and a '
to distinguish from the originals). Our previous observation leads
to
\[
2{s_1}'+{s_2}'+{s_3}'+{s_4}'=4n-3+{x_a}'+{y_{\overline a}}'+{z_{\overline a}}'+{w_a}'.
\]

The order of magnitude of Gy\'arf\'as' original 
bound is already proven by our observations.
We see some additional terms too, but the usage of them requires
a case analysis.

\noindent{\bf Case~1:\ }
$s_1+s_a'\geq 2n-2$.
In this case the longer of $S_1$ and ${S_1}'$
proves Gy\'arf\'as' conjecture and we are done.

\noindent{\bf Case~2:\ }
$s_1+{s_a}'< 2n-2$. 

$S_1$ is a \jl-staircase, $s_1$ is its size.
Let $s_1^-$ denote the size of the horizontal hand of $S_1$,
and let $s_1^\medvert$ denote the size of
its vertical hand. We introduce similar notation for its symmetric
pair, the staircase, broken at the left bottom element of $M$  
(its size is denoted by ${s_a}'$).

Since
$s_1=s_1^-+s_1^\medvert-1$ and
${s_a}'={s_a^-}'+{s_a^\medvert}'-1$, we have 
\[
2n-2>s_1+{s_a}'=(s_1^-+s_1^\medvert-1)+({s_a^-}'+{s_a^\medvert}'-1)
=(s_1^-+{s_a^-}')+(s_1^\medvert+{s_a^\medvert}')-2.
\]

Without loss of generality we can assume, that 
$s_1^-+{s_a^-}'<n$.
It is obvious that we are guaranteed to have
a column in our matrix that has a $0$
in the first/top position, $a_5$ and has an $\overline a$
at the last/bottom position, $a_6$.
In the next picture we shaded this column and extended it
to a staircase shape with further
shaded positions (we are talking about staircase SHAPE,
not about staircase!):
\begin{equation*}
\left(\begin{array}{ccccccccccc}
\lep &\lep &\lep &\lep &\lep \cdots 
                         &\lep 0_{a_5}      &\cdots&0_{a_2}& 1 & \cdots &  1   \\
            &  &  &  &      
                         &\lep        &   & & &       &                \\
\vdots      &  &  &  &      
                         &\lep\vdots  &   & & &       & \vdots         \\
            &  &  &  &      
                         &\lep        &   & & &       &                \\
a & \cdots & a&{\overline a}_{{a_3}'}       & \cdots 
                         &\lep\overline{a}_{a_6} &\lep\cdots   &\lep &\lep &\lep &\lep\\
  \end{array}\right)
\end{equation*}

We define two staircases among the shaded elements.
We distinguish two subcases:

\noindent{\bf Subcase~1:}
$0=\overline a$.
Let $S_5$ be the sequence of shaded $0$'s and
let $S_6$ be the $1$'s of the shaded column.

\noindent{\bf Subcase~2:}
$0\not=\overline a$.
Let $S_5$ be a \jl-staircase centered at $a_5$, and
Let $S_6$ be a \lj-staircase centered at $a_6$.

In both cases $|S_5|+|S_6|$ counts all shaded
elements (that is $2n-1$ elements)
except the shaded $1$'s in the first row and the shaded $a$'s in the last row.
We can upper bound their number with $x_1+{x_a}'$.
We obtained that
\[
|S_5|+|S_6|\geq 2n-1-x_1-{x_a}'.
\]
Hence
\[
2s_1+s_2+s_3+s_4+
2{s_1}'+{s_2}'+{s_3}'+{s_4}'+|S_5|+|S_6|
\geq
10n-7+y_0+z_0+w_1+{y_{\overline a}}'+{z_{\overline a}}'+{w_a}'.
\]
We investigated $10$ staircases. 
A weighted average of their length is at least
$1/12(10n-7)=5n/6-7/12$. So the longest staircase 
in our proof proves the theorem.

\begin{remark}
Before we started the main streamline of the proof we excluded
$a_4=0$ (and later we also excluded ${a_4}'=\overline a$). 
There, we used staircases $S$ such that $\turn(S)=3$.
All the other staircases in our proof have turning points at most $2$.
So our theorem can be stated in a stronger form:
{\it Any $M\in\set{0,1}^{n\times n}$ contains a staircase
that has length at least $5n/6-7/12$, and it has at most $3$ turning points.}

Gy\'arf\'as exhibited an example that shows that his
$4n/5$ cannot be improved if we use only staircases
with at most $1$ turning point.

We do not know any bound for the case when we
are allowed to use only staircases with at most $2$ turning points.
\end{remark}

\enddocument